\documentclass[12pt,a4paper]{amsart}
\usepackage[width=0.75\paperwidth,height=0.85\paperheight,twoside=false,includehead,includefoot]{geometry}
\allowdisplaybreaks[1]

\DeclareFontEncoding{OT2}{}{}
\DeclareFontSubstitution{OT2}{cmr}{m}{n}

\DeclareFontFamily{OT2}{cmr}{\hyphenchar\font45 }
\DeclareFontShape{OT2}{cmr}{m}{n}{%
   <5><6><7><8><9>gen*wncyr%
   <10><10.95><12><14.4><17.28><20.74><24.88>wncyr10}{}

\DeclareMathAlphabet{\mathcyr}{OT2}{cmr}{m}{n}
\DeclareMathAlphabet{\mathcyb}{OT2}{cmr}{b}{n}
\SetMathAlphabet{\mathcyr}{bold}{OT2}{cmr}{b}{n}

\usepackage{amsopn}

\usepackage{amstext}
\usepackage{amsthm}
\usepackage{amssymb}

\newtheorem{thm}{Theorem}[section]
\newtheorem{lem}[thm]{Lemma}
\newtheorem{prop}[thm]{Proposition}

\theoremstyle{definition}

\theoremstyle{remark}
\newtheorem{rem}[thm]{Remark}

\begin{document}

\title{On the $a$-points of symmetric sum of multiple zeta function}

\author{Hideki Murahara}
\address[Hideki Murahara]{Nakamura Gakuen University Graduate School, 5-7-1, Befu, Jonan-ku,
Fukuoka, 814-0198, Japan}
\email{hmurahara@nakamura-u.ac.jp}

\author{Tomokazu Onozuka}
\address[Tomokazu Onozuka]{Institute of Mathematics for Industry, Kyushu University 744, Motooka, Nishi-ku,
Fukuoka, 819-0395, Japan}
\email{t-onozuka@imi.kyushu-u.ac.jp}

\subjclass[2010]{Primary 11M32}
\keywords{Multiple zeta function, the Riemann zeta function, the Riemann-von Mangoldt formula, a-point.}

\begin{abstract}
 In this paper, we present some results on the $a$-points of the symmetric sum of the Euler-Zagier multiple zeta function. 
 Our first three results are for the $a$-points free region of the function. 
 The fourth result is the Riemann-von Mangoldt type formula. 
 In the last two results, we study the real parts of $a$-points of the function. 
\end{abstract}

\maketitle

\section{Introduction}
With regard to the Riemann hypothesis, research on non-trivial zeros is being actively conducted.
Among them, Riemann-von Mangoldt's formula $N (T)=T/2 \pi \log T/2 \pi -T/2 \pi + O(\log T)$ is well known,
where $N(T)$ is the number of non-trivial zeros of $\zeta (s)$ with $0<\Im(s)<T$. 
This formula was generalized from zeros to $a$-points by Landau in \cite{BLL13}, 
where $a$-points of the function $f(s)$ are points defined by the solutions of $f (s)=a$. 
Note that $a$-points are zeros if $a=0$.

On the other hand, Bohr and Landau in \cite{BL14} showed that almost all non-trivial zeros of the Riemann zeta function are near the critical line.
Levinson in \cite{Lev75} extended this result to $a$-points, indicating that almost all $a$-points of the Riemann zeta function are near the critical line. 
Furthermore, some of these results have been generalized to derivatives of the Riemann zeta function (\cite{Ber70}, \cite{LM74}, \cite{Ono17}). 

In this paper, we prove some results on the $a$-points for symmetric sums with the Euler-Zagier multiple zeta functions (MZFs).  
The MZF is defined by 
\begin{align*}
 \zeta(s_{1},\dots,s_{r}):=\sum_{1\le n_{1}<\cdots<n_{r}}\frac{1}{n_{1}^{s_{1}}\cdots n_{r}^{s_{r}}},
\end{align*}
where $s_{j}\in\mathbb{C}\,(j=1,\ldots,r)$ are complex variables.
Matsumoto in \cite{Mat02} proved that the series is absolutely convergent
in the domain 
\[
 \{(s_{1},\ldots,s_{r})\in\mathbb{C}^{r}\;|\;\Re(s(l,r))>r-l+1\;(1\leq l\leq r)\},
\]
where $s(l,r):=s_{l}+\cdots+s_{r}$. 
Akiyama, Egami, and Tanigawa in \cite{AET01} and Zhao in \cite{Zha00} independently proved that $\zeta(s_{1},\dots,s_{r})$
is meromorphically continued to the whole space $\mathbb{C}^{r}$.
Kamano in \cite{Kam06} mentioned the trivial zeros of $\zeta(s,\dots,s)$. 
Non-trivial zeros of MZFs are numerically studied by Matsumoto and Sh\={o}ji in \cite{MS14,MS20}. 
Nakamura and Pa\'{n}kowski in \cite{NP16} estimated the number of $a$-points of $\zeta(s,\dots,s)$. 
Ikeda and Matsuoka in \cite{IM} studied zeros of $\zeta(s,\dots,s)$.

Throughout this paper, we write $s:=\sigma+it$ and $s_j:=\sigma_j+it_j$ for $j=1,\dots,r$.
Let $\mathfrak{S}_r$ be a symmetric group of degree $r$. 
We consider the function $\widetilde\zeta(s)$ defined by
\[
 \widetilde\zeta(s)
:=\sum_{\tau\in\mathfrak{S}_r} \zeta(a_{\tau(1)}s,\ldots,a_{\tau(r)}s)
\]
for $a_1,\ldots,a_r\in\mathbb{R}_{>0}$. 
Since $\widetilde\zeta(s)$ is a generalization of $\zeta (s,\dots,s)$ and is one of the simplest functions that can be created by multiplexing the Riemann zeta function, to study this function is a starting point for considering the $a$-point of MZFs.
In \cite{Hof92}, Hoffman showed 
\[
 \widetilde\zeta(s)
 =\sum_{l=1}^{r} 
  \sum_{P_1,\dots, P_l} 
  (-1)^{r-l} \prod_{j=1}^{l} (\#(P_j)-1)! \cdot \zeta\biggl( \sum_{j'\in P_j} a_{j'}s \biggr),  
\]
where $P_1,\dots, P_l$ runs the all partition of $\{1,\ldots,r \}$.
In the sequel, assume that $r\ge2$ is fixed, and $a_1,\dots,a_r$ satisfy $a_1\ge\cdots\ge a_r>0$.
\begin{thm} \label{main1}
For $a\in\mathbb{C}$, there exists a constant $C_1>0$ such that 
$\widetilde\zeta(s)$ has no $a$-point for $\sigma>C_1$. 
\end{thm}

Let $A:=a_1+\cdots+a_r$.
We define 
\begin{align*}
C(\epsilon)
&:=\bigcup_{n=1}^{\infty} \left\{
s\in\mathbb{C} \mid \left|s+\frac{ 2n }{ A } \right|\le\epsilon 
\right\}, \\
 D(C_2,p)
 &=\left\{
   s\in\mathbb{C} \mid \sigma \le \min\{ C_2,-pt^2 \}
  \right\}.
\end{align*}
\begin{thm} \label{main2}
 For $a\in\mathbb{C}$, $p>0$, and small $\epsilon>0$, there exists a constant $C_2<0$ such that 
 $\widetilde\zeta(s)$ has no $a$-point in $D(C_2,p) \setminus C(\epsilon)$. 
 In addition, for each disk there exists exactly one $a$-point.
\end{thm}

\begin{thm} \label{main2+}
 For $a\in\mathbb{C}$ and $y_1>y_2>0$, there exists a constant $C_3>0$ such that 
 $\widetilde\zeta(s)$ has no $a$-point in 
 $\{ s\in\mathbb{C} \mid -y_1\le \sigma \le -y_2, |t|>C_3 \}$.
\end{thm}

Let $\rho_{a}:=\beta_{a}+i\gamma_{a}$ be an $a$-point of $\widetilde\zeta(s)$. 
For $a\in\mathbb{C}$, let $N_{y}(a;T)$ count the number of $a$-points $\rho_{a}$ with multiplicity of $\widetilde{\zeta}(s)$ with $-y<\beta_a$ and $0<\gamma_a<T$. 
\begin{thm} \label{main3}
 For $a\in\mathbb{C}$ and $y>0$, we have
 \begin{align*}
  &N_{y}(a;T) \\
  &=
  \begin{cases}
   \displaystyle
   \frac{T}{2\pi} \sum_{ j=1 }^r a_j \log \frac{a_jT}{2\pi}-\frac{AT}{2\pi} 
   -\frac{ T }{ 2\pi }\log(1^{a_1}\cdots r^{a_r})+O(\log T) \quad(a=0), \\
   \displaystyle
   \frac{T}{2\pi} \sum_{ j=1 }^r a_j \log\frac{a_jT}{2\pi}-\frac{AT}{2\pi} 
   +O(\log T)  \qquad\qquad\qquad\qquad\quad\,\,\,\, (a\ne0).
  \end{cases}
 \end{align*}
\end{thm}
\begin{thm} \label{main4}
 Let $y>0$. 
 For large $T$, we have 
 \begin{align*}
  2\pi \sum_{\substack{ -y < \beta_a \\ 0<\gamma_a <T }} 
  \left( \beta_a -\frac{ 1 }{ 2 } \right)
  &=\frac{ r-A }{ 2 }\, T\log T+O(T), \\
  2\pi \sum_{\substack{ -y < \beta_a \\ 0<\gamma_a <T }} 
  \left( \beta_a -\frac{ r }{ 2A } \right)
  &=O(T). 
 \end{align*}
\end{thm}
\begin{thm} \label{main5}
 Let $y_3=1/(2a_r)$. 
 We have 
 \[
  2\pi \sum_{\substack{ y_3+\delta < \beta_a \\ 0<\gamma_a <T }} 1
  =O\left( \frac{ T\log\log T }{ \delta }\right) 
 \]
 uniformly for $\delta>0$. 
 Especially, we have 
 \[
  2\pi \sum_{\substack{ y_3+\frac{ (\log\log T)^2 }{ \log T } < \beta_a \\ 0<\gamma_a <T }} 1
  =O\left( \frac{ T\log T }{ \log\log T }\right). 
 \]
\end{thm}

\begin{rem}
We do not know that $y_3$ is the best possible.
\end{rem}

Matsumoto and Tsumura in \cite{MT15} gave the mean value theorems for the double zeta function $\zeta(s_1,s_2)$ with respect to $s_2$, and they mentioned that the region $\{(s_1,s_2)\ |\ \sigma_1+\sigma_2=3/2\}$ might be the double analogue of the critical line of the Riemann zeta function. Ikeda, Matsuoka, and Nagata in \cite{IMN} also calculated the mean value $\int|\zeta(\sigma_1+it,\sigma_2+it)|^2dt$, and in this case, they conjectured that the boundary of the region $\{(s_1,s_2)\ |\ \sigma_1+\sigma_2>1,\ \sigma_2>1/2\}$ is  the double analogue of the critical line. These analogues comes from the the mean values of the double zeta function. For the $a$-points of the Riemann zeta function, by \cite[Lemma 5]{Lev75} and \cite[(22)]{BLL13}, we can deduce
\begin{align*}
2\pi \sum_{ T<\gamma_{\zeta,a} <T+U } \left( \beta_{\zeta,a} -\frac{ 1 }{ 2 } \right)
&=2\pi\sum_{T<\gamma_{\zeta,a}<T+U}(\beta_{\zeta,a}+b)-2\pi\left(b+\frac{1}{2}\right)\left\{N_{\zeta,a}(T+U)-N_{\zeta,a}(T)\right\}\\
&=O(U),
\end{align*}
where $b\ge2$, $T^{1/2}\le U\le T$, $\rho_{\zeta,a}=\beta_{\zeta,a}+i\gamma_{\zeta,a}$ is an $a$-point of the Riemann zeta function $\zeta(s)$, and $N_{\zeta,a}(T)$ is the number of $a$-points of $\zeta (s)$ with $1<\Im(s)<T$. This estimate means that $a$-points of $\zeta(s)$ are evenly distributed on the left and right concerning the critical line $\sigma=1/2$, since the main terms of $2\pi\sum_{T<\gamma_{\zeta,a}<T+U}(\beta_{\zeta,a}+b)$ and  $2\pi\left(b+\frac{1}{2}\right)\{N_{\zeta,a}(T+U)-N_{\zeta,a}(T)\}$ cancel each other out. 
This is one of the special features of the critical line. From this point of view, we conjecture that $\sigma=r/(2A)$ is an analogue of the critical line of the Riemann zeta function for $\widetilde\zeta(s)$ by Theorem \ref{main4}.

\section{Proof of Theorem \ref{main1}}
\begin{lem} \label{21}
 For $z_1,\dots,z_r\in\mathbb{C}$ with $\Re(z_1),\dots,\Re(z_r)>2$ and $m\in\mathbb{Z}_{>0}$, we have 
 \begin{align*}
  &\sum_{m\le m_1<\cdots<m_r} \frac{ 1 }{ m_1^{z_1}\cdots m_r^{z_r} } \\
  &=
   \begin{cases}
    \displaystyle{
    \frac{ 1 }{ m^{z_1} } +O\left(\frac{ 1 }{ (m+1)^{\Re(z_1)-1} } \right)
    } \qquad (r=1) \\
    \displaystyle{
    \frac{ 1 }{ m^{z_1} \cdots (m+r-1)^{z_r} } 
    +O\left(\frac{ 1 }{ m^{\Re(z_1)} \cdots (m+r-2)^{\Re(z_{r-1})}\cdot (m+r)^{\Re(z_r)-r} }\right) 
    \qquad (r\ge2)
    }. 
   \end{cases}
 \end{align*}
\end{lem}
\begin{proof}
 We prove this lemma by induction on $r$. 
 When $r=1$, we have 
 \begin{align*}
  \sum_{m_1=m}^{\infty} \frac{ 1 }{ m_1^{z_1} }
  &=\frac{ 1 }{ m^{z_1} } + \frac{ 1 }{ (m+1)^{z_1} } +O\left(\int_{m+1}^\infty \frac{ dx }{ x^{z_1} } \right) \\
  &=\frac{ 1 }{ m^{z_1} } +O\left(\frac{ 1 }{ (m+1)^{\Re(z_1)-1} } \right).
 \end{align*}
 In the general case, by the induction hypothesis, we have 
 \begin{align*}
  &\sum_{m\le m_1<\cdots<m_r} \frac{ 1 }{ m_1^{z_1}\cdots m_r^{z_r} } \\
  &=\sum_{m_1=m}^{\infty} \frac{ 1 }{ m_1^{z_1} }
   \sum_{m_1+1\le m_2<\cdots<m_r} \frac{ 1 }{ m_2^{z_2}\cdots m_r^{z_r} } \\
  &=\sum_{m_1=m}^{\infty} 
   \biggl(
   \frac{ 1 }{ m_1^{z_1} \cdots (m_1+r-1)^{z_r} } 
   +O\biggl(\frac{ 1 }{ m_1^{\Re(z_1)} \cdots (m_1+r-2)^{\Re(z_{r-1})}\cdot (m_1+r)^{\Re(z_r)-r+1} }\biggr)
   \biggr).
 \end{align*}
 We notice that
 \begin{align*}
 &\sum_{m_1=m}^{\infty} 
  \frac{ 1 }{ m_1^{z_1} \cdots (m_1+r-1)^{z_r} } \\
 &=\frac{ 1 }{ m^{z_1} \cdots (m+r-1)^{z_r} }
  +\cdots
  +\frac{ 1 }{ (m+r)^{z_1} \cdots (m+2r-1)^{z_r} } \\
 &\quad +O\left(
  \int_{m+r}^\infty \frac{ 1 }{ x^{\Re(z_1)}\cdots (x+r-1)^{\Re(z_r)} } dx
  \right)
 \end{align*}
 holds. 
 Here we see that
 \begin{align*}
  \int_{m+r}^\infty \left| \frac{ 1 }{ x^{z_1} \cdots (x+r-1)^{z_r} } \right| dx
  &\le \int_{m+r}^\infty \left| \frac{ 1 }{ x^{z_1+\cdots +z_r} } \right| dx \\
  &\le \frac{ 1 }{ (m+r)^{\Re(z_1)+\cdots+\Re(z_r)-1} }.
 \end{align*}
 From the above equalities, we get 
 \begin{align*}
 &\sum_{m_1=m}^{\infty} 
  \frac{ 1 }{ m_1^{z_1} \cdots (m_1+r-1)^{z_r} } \\
  &=\frac{ 1 }{ m^{z_1} \cdots (m+r-1)^{z_r} } 
  +O\left(\frac{ 1 }{ m^{\Re(z_1)} \cdots (m+r-2)^{\Re(z_{r-1})}\cdot (m+r)^{\Re(z_r)-r} }\right). 
 \end{align*}
 Similarly, we have  
 \begin{align*}
  &\sum_{m_1=m}^{\infty} 
   \frac{ 1 }{ m_1^{\Re(z_1)} \cdots (m_1+r-2)^{\Re(z_{r-1})}\cdot (m_1+r)^{\Re(z_r)-r+1} } \\
  &=O\biggl( \frac{ 1 }{ m^{\Re(z_1)} \cdots (m+r-2)^{\Re(z_{r-1})}\cdot (m+r)^{\Re(z_r)-r} } \biggr). \qedhere
 \end{align*}
\end{proof}

Let
\begin{align*}
 B&:=\# \{ \tau\in\mathfrak{S}_r \mid (a_{\tau(1)},\dots,a_{\tau(r)}) =(a_1,\dots,a_r) \}, \\
 M&:=\frac{1}{1^{a_1}\cdots r^{a_r}}. 
\end{align*}
\begin{lem} \label{22}
 For $s\in\mathbb{C}$ with $a_r\sigma>2$, there exists a constant $0<c<1$ such that 
 \begin{align*} 
  \widetilde\zeta(s)
  =BM^s (1+O(c^\sigma)).
 \end{align*}
\end{lem}
\begin{proof}
 By Lemma \ref{21}, we have 
 \begin{align*}
  \widetilde\zeta(s) 
  &=\sum_{\tau\in\mathfrak{S}_r} \frac{ 1 }{ 1^{a_{\tau(1)}s} \cdots r^{a_{\tau(r)}s} } 
   +\sum_{\tau\in\mathfrak{S}_r} 
   O\left(\frac{ 1 }{ 1^{a_{\tau(1)}\sigma} \cdots (r-1)^{a_{\tau(r-1)}\sigma}\cdot (r+1)^{a_{\tau(r)}\sigma-r} }\right). 
 \end{align*}
 Since
 \[
  0<\frac{ 1^{a_{1}} \cdots r^{a_{r}} }{ 1^{a_{\tau(1)}} \cdots r^{a_{\tau(r)}}}<1
 \]
 holds if $(a_{\tau(1)},\dots,a_{\tau(r)})\ne(a_1,\dots,a_r)$, we have the desired result. 
\end{proof}

\begin{proof}[Proof of Theorem \ref{main1}]
 By Lemma \ref{22}, there exists a sufficiently large $C_1$ such that $|\widetilde\zeta(s)/(BM^s) -1|<1/2$ for $\sigma>C_1$. 
 This fact implies the theorem for $a=0$.  
 As $\sigma\rightarrow\infty$, $M^s$ tends to $0$.
 Then, for $a\ne0$, there exists $C_1$ such that 
 $|\widetilde\zeta(s)|<|a|/2$ for $\sigma>C_1$, which implies the theorem. 
\end{proof}

\section{Proof of Theorem \ref{main2}} 
Let
\begin{align*}
 C'(\epsilon)
 &:=\bigcup_{n=1}^{\infty} 
  \left\{
   s\in\mathbb{C} \mid \left|s+2n \right|\le\epsilon 
  \right\}, \\
 D'(p)
 &=\left\{
  (s_1,s_2)\in\mathbb{C}^2 \mid |s_1|\ge|s_2|, \sigma_1\le-pt_1^2, \sigma_2\le-pt_2^2, t_1\ge0,t_2\ge0 
  \right\}.
\end{align*}
\begin{lem} \label{3.2}
 Let $\epsilon>0$ and $p>0$.
 For $(s_1,s_2)\in D'(p)$ with $\sigma_1, \sigma_2$ sufficiently small and $s_1+s_2\notin C'(\epsilon)$, there exists $c>1$ such that  
 \[
  \frac{\zeta(s_1+s_2)}{\zeta(s_1)\zeta(s_2)} 
  \gg c^{-\sigma_2}.
 \]
\end{lem}
\begin{proof}
 We first note that
 \begin{align*}
  \zeta(s)
  &=2^s\pi^{s-1}\sin\left(\frac{\pi s}{2}\right)\Gamma(1-s) \zeta(1-s) \\
  &=2^{1/2+s}\pi^{s-1/2} e^{s-1} (1-s)^{1/2-s} \sin\left(\frac{\pi s}{2}\right) \zeta(1-s) \left(1+O\left(\frac{1}{|1-s|}\right)\right) 
 \end{align*}
 holds by the well-known functional equation and the Stirling formula. 
 Let
 \begin{align*}
  \alpha_1
  &:=\prod_{j=1}^{2} \left( |1-s_j|^{1/2-\sigma_j} \exp(t_j (\arg(1-s_j)+\pi/2)) \right), \\
  \alpha_2
  &:=|1-s_1-s_2|^{1/2-\sigma_1-\sigma_2} \exp((t_1+t_2) (\arg(1-s_1-s_2)+\pi/2)).  
 \end{align*}
 By the assumption $s_1+s_2\notin C'(\epsilon)$, we have 
 \[
  \frac{\zeta(s_1+s_2)}{\zeta(s_1)\zeta(s_2)}
  \gg \frac{ \alpha_2 }{ \alpha_1 }. 
 \]
 Since there exists $c>1$ such that $|1-s_1-s_2|\ge c|1-s_2|$, we have
 \begin{align} \label{111}
  \begin{split}
   \frac{ |1-s_1-s_2|^{1/2-\sigma_1-\sigma_2} }{ |1-s_1|^{1/2-\sigma_1} |1-s_2|^{1/2-\sigma_2} } 
   &=\left| 1+ \frac{ -s_2 }{ 1-s_1 } \right|^{1/2-\sigma_1} \frac{ | 1-s_1-s_2 |^{-\sigma_2} }{ | 1-s_2 |^{1/2-\sigma_2} } \\
   &\ge \left| 1+ \frac{ -s_2 }{ 1-s_1 } \right|^{1/2-\sigma_1} \frac{ c^{-\sigma_2} }{ | 1-s_2 |^{1/2} }.
  \end{split}
 \end{align}
 On the other hand, since $\sigma_1\le-pt_1^2$ and $\sigma_2\le-pt_2^2$, we have
 \begin{align} \label{112}
  \begin{split}
   &\exp\left( (t_1+t_2)\arg(1-s_1-s_2) -t_1\arg(1-s_1) -t_2\arg(1-s_2) \right) \\
   &=\exp\left( 
    -(t_1+t_2)\arctan \frac{ t_1+t_2 }{ 1-\sigma_1-\sigma_2 } 
    +t_1\arctan \frac{ t_1 }{ 1-\sigma_1 } 
    +t_2\arctan \frac{ t_2 }{ 1-\sigma_2 }
    \right) \\
   &=\exp\biggl( 
    -(t_1+t_2) \frac{ t_1+t_2 }{ 1-\sigma_1-\sigma_2 } 
    +t_1 \frac{ t_1 }{ 1-\sigma_1 } 
    +t_2 \frac{ t_2 }{ 1-\sigma_2 } \\
    &\qquad\quad 
     +O\biggl( \frac{ t_1^4 }{ (1-\sigma_1)^3 } \biggr)
     +O\biggl( \frac{ t_2^4 }{ (1-\sigma_2)^3 } \biggr)
     +O\biggl( \frac{ (t_1+t_2)^4 }{ (1-\sigma_1-\sigma_2)^3 } \biggr)
   \biggr) \\
   &=\exp (O(1)).
  \end{split}
 \end{align}
 By \eqref{111} and \eqref{112}, we have
 \begin{align*}
 \frac{ \alpha_2 }{ \alpha_1 } 
  &=\frac{ |1-s_1-s_2|^{1/2-\sigma_1-\sigma_2} }{ |1-s_1|^{1/2-\sigma_1} |1-s_2|^{1/2-\sigma_2} } \\
   &\quad \times \exp\left( (t_1+t_2)\arg(1-s_1-s_2) -t_1\arg(1-s_1) -t_2\arg(1-s_2) \right) \\
  &\gg \left| 1- \frac{ s_2 }{ 1-s_1 } \right|^{1/2-\sigma_1} \frac{ c_0^{-\sigma_2} }{ | 1-s_2 |^{1/2} } 
 \end{align*}
 for $c_0>1$.
 Since
 \[
  \frac{-s_2}{1-s_1}
  =\frac{-s_2 (1-\overline{s_1})}{(1-\sigma_1)^2+t_1^2}
 \]
 and 
 \[
  -\frac{\pi}{2} < \arg (-s_2 (1-\overline{s_1})) <\frac{\pi}{2},
 \]
 we have
 \[
  \left| 1- \frac{ s_2 }{ 1-s_1 } \right|>1.
 \]
 There exist $c>1$ such that
 \[
  c^{-\sigma_2}>\frac{ c_0^{-\sigma_2} }{ | 1-s_2 |^{1/2} }.
 \]
 This finishes the proof. 
\end{proof}

\begin{lem} \label{1111}
 For $p>0$ and small $\epsilon>0$, there exists a constant $c<0$ such that 
 \begin{align*}
  \widetilde\zeta(s)
  &\asymp \zeta(As)
 \end{align*}
 for $s\in D(c,p) \setminus C(\epsilon)$. 
\end{lem}
\begin{proof}
 By applying the previous lemma repeatedly, we have the result. 
\end{proof}
\begin{proof}[Proof of Theorem \ref{main2}]
 For $s\in D(c,p) \setminus C(\epsilon)$ with $t\ge0$, we have 
 \begin{align*}
  |\widetilde\zeta(s)|
  &\gg |\zeta(As)| \\
  &\gg |1-As|^{1/2-A\sigma} \exp \left(At \left(\arg(1-As)+\frac{ \pi }{ 2 }\right)\right) \\
  &\gg |1-As|^{1/2-A\sigma}. 
 \end{align*} 
 Then, for sufficiently small $\sigma$, we get $|\widetilde\zeta(s)|>|a|$. 

 Applying Rouch\'e's theorem for $\zeta(As)$ and $\widetilde\zeta(s)-\zeta(As)-a$, we find that the functions 
 $\widetilde\zeta(s)-a$ and $\zeta(As)$ have the same number of zeros in each disk of $C(\epsilon)$. 
 Since the function $\zeta(As)$ has exactly one zero in each disk, we find the result. 
\end{proof}

\section{Proof of Theorems \ref{main2+} and \ref{main3}} 
For $s\in\mathbb{C}$, we define
\begin{align*}
 G(s):=
 \begin{cases}
  \displaystyle{ \frac{ \widetilde\zeta(s) }{ BM^s }} &(a=0), 
  \vspace{0.5em} \\
  \displaystyle{ \frac{ \widetilde\zeta(s)-a }{ -a }} &(a\ne0).
 \end{cases}
\end{align*}

Let $x>0$ be large.
For large $T$, let each side of positive oriented rectangle with vertices $-y+it_0,x+it_0,x+iT,-y+iT$ be $R_1,R_2,R_3,R_4$ counting from the bottom. 
We note that 
\begin{align} \label{note}
 \begin{split} 
  \frac{ G'(s) }{ G(s) }=
  \begin{cases}
   \displaystyle{ \frac{ \widetilde\zeta'(s) }{ \widetilde\zeta(s) }-\log M} &(a=0), 
   \vspace{0.5em} \\
   \displaystyle{ \frac{ \widetilde\zeta'(s) }{ \widetilde\zeta(s)-a }} &(a\ne0). 
  \end{cases}
 \end{split}
\end{align}
We also note that
$\widetilde\zeta(s)=a$ if and only if $G(s)=0$. 
We prove Theorem \ref{main3} by using the following formula: 
\begin{align*}
 N_{y}(a;T)
 =\frac{ 1 }{ 2\pi i } \left( \int_{R_1} +\int_{R_2} +\int_{R_3} +\int_{R_4} \right) \frac{ G'(s) }{ G(s) } ds 
  +O(1).
\end{align*}

\subsection{Calculations for $R_3$} 
Let $q(x,u)$ be the number of times that the value of $\Re(G(z))$ is $0$ when the variable $z$ moves from $x+iT$ to $u+iT$ parallel to the real axis, and for $z\in\mathbb{C}$, let  
\[
 X(z):=X_T(z):=\frac{ G(z+iT) +G(z-iT) }{ 2 }.
\]
\begin{lem}
 For $u,T\in\mathbb{R}$ with $u<x$, we have   
 \[
  |\arg G(u+iT)|\le (q(x,u)+2)\pi.
 \]
\end{lem}
\begin{proof} 
 Since $|\arg G(x+iT)|\le\pi/2$, this lemma is derived by considering the argument.
\end{proof}
We also let $n(K)$ be the number of zeros of $X(z)$ inside the circle with center $x$ and radius $K$.
Since $X(u)=\Re (G(u+iT))$,  we have $q(x,-y)\le n(x+y)$.  
\begin{lem}
 We have 
 \[
  q(x,-y)\le (x+y+2)\int_{0}^{x+y+2} \frac{ n(r) }{ r } dr.
 \]
\end{lem}
\begin{proof}
 We have 
 \begin{align*}
  \int_{0}^{x+y+2} \frac{ n(r) }{ r } dr 
  &\ge \int_{x+y+1}^{x+y+2} \frac{ n(r) }{ r } dr \\
  &\ge \frac{ 1 }{ x+y+2 } \int_{x+y+1}^{x+y+2} n(r) dr \\
  &\ge \frac{ n(x+y+1) }{ x+y+2 } \\
  &\ge \frac{ q(x,-y) }{ x+y+2 }. \qedhere
 \end{align*}
\end{proof}
\begin{lem}
 For $a\in\mathbb{C}$ and $y>0$, we have
 \[
  \int_{0}^{x+y+2} \frac{ n(r) }{ r } dr
  \ll \log T.
 \]
\end{lem}
\begin{proof}
 By Jensen's theorem, we have 
 \[
  \int_{0}^{x+y+2} \frac{ n(r) }{ r } dr
  = \frac{ 1 }{ 2\pi } \int_{0}^{2\pi} \log|X(x+(x+y+2)e^{i\theta})| d\theta
   -\log|X(x)|.
 \]
 We note that
 \[
  |X(x+(x+y+2)e^{i\theta})|
  \le |G(x+(x+y+2)e^{i\theta}+iT)| +|G(x+(x+y+2)e^{i\theta}-iT)|.
 \] 
 For each $\sigma$, we define $\mu(\sigma)$ as the lower bound of $\xi$ such that 
 \[
  \zeta(\sigma+it)=O(|t|^\xi).
 \] 
 It is well-known that
 \begin{align*}
  \mu(\sigma)\le
  \begin{cases}
   0 &(\sigma>1), \\
   1/2-\sigma/2 &(0\le\sigma\le 1), \\
   1/2-\sigma &(\sigma<0)   
  \end{cases}
 \end{align*}
 holds (see \cite[Section 5.1]{Tit86}).
 Since there exist some $L$ such that
 \[
  \widetilde\zeta(x+(x+y+2)e^{i\theta}+iT)
  \ll T^L,
 \]
 we have 
 \begin{align*}
  |G(x+(x+y+2)e^{i\theta}+iT)|
  \ll T^L.
 \end{align*}
 Then we get
 \[
  |X(x+(x+y+2)e^{i\theta}+iT)|
  \ll T^L. 
 \]
 Thus we find the lemma. 
\end{proof}
\begin{lem} \label{mo3b}
 For $a\in\mathbb{C}$ and $y>0$, we have
 \begin{align*}
  \Im\int_{R_3} \frac{ G'(s) }{ G(s) } ds 
  \ll \log T.
 \end{align*}
\end{lem}
\begin{proof}
 By above lemmas, we have 
 \begin{align} \label{sekibunmae}
  |\arg G(u+iT)| 
  \ll \log T
 \end{align}
 for $-y\le u\le x$.
 We have 
 \begin{align*}
  \Im\int_{R_3} \frac{ G'(s) }{ G(s) } ds
  &=\Im\int_{x}^{-y} \frac{ G'(\sigma+iT) }{ G(\sigma+iT) } d\sigma \\
  &=[\arg G(\sigma+iT)]_{x}^{-y} \\
  &\ll \log T. \qedhere
 \end{align*}
\end{proof}

\subsection{Calculations for $R_4$} 
\begin{lem} 
 Let $b_1,b_2>0$ and $y_1>y_2>0$. 
 For $s\in\mathbb{C}$ with $-y_1<\sigma<-y_2$ and sufficiently large $t$, we have 
 \[
  \frac{\zeta(b_1s)\zeta(b_2s)}{\zeta((b_1+b_2)s)} 
  \gg t^{1/2}.
 \]
\end{lem}
\begin{proof}
Let $s_1:=b_1s$ and $s_2:=b_2s$.
 The proof is almost the same as the proof of Lemma \ref{3.2}.
 By letting
 \begin{align*}
  \alpha_1
  &:=\prod_{j=1}^{2} \left( |1-s_j|^{1/2-\sigma_j} \exp(t_j (\arg(1-s_j)+\pi/2)) \right), \\
  \alpha_2
  &:=|1-s_1-s_2|^{1/2-\sigma_1-\sigma_2} \exp((t_1+t_2) (\arg(1-s_1-s_2)+\pi/2)), 
 \end{align*}
 we have 
 \[
  \frac{\zeta(s_1)\zeta(s_2)}{\zeta(s_1+s_2)}
  \asymp \frac{ \alpha_1 }{ \alpha_2 }. 
 \]
 We note that
 \begin{align} \label{4441}
  \begin{split}
   \frac{ |1-s_1|^{1/2-\sigma_1} |1-s_2|^{1/2-\sigma_2} }{ |1-s_1-s_2|^{1/2-\sigma_1-\sigma_2} } 
   \asymp \frac{ t_1^{1/2-\sigma_1} t_2^{1/2-\sigma_2} }{ (t_1+t_2)^{1/2-\sigma_1-\sigma_2} }
   \asymp t^{1/2}.
  \end{split}
 \end{align}
 On the other hand, we have
 \begin{align} \label{4442}
  \begin{split}
   &\exp\left( t_1(\arg(1-s_1) +\pi/2) +t_2(\arg(1-s_2) +\pi/2) -(t_1+t_2)(\arg(1-s_1-s_2) +\pi/2) \right) \\
   &=\exp \biggl( 
    t_1\left( \arctan \left( \frac{ -t_1 }{ 1-\sigma_1 } \right) +\frac{ \pi }{ 2 } \right)
    +t_2\left( \arctan \left( \frac{ -t_2 }{ 1-\sigma_2 } \right) +\frac{ \pi }{ 2 } \right) \\
    &\qquad\quad -(t_1+t_2) \left( \arctan \left( \frac{ -t_1-t_2 }{ 1-\sigma_1-\sigma_2 } \right) +\frac{ \pi }{ 2 } \right) 
    \biggr) \\
   &=\exp \biggl( 
    t_1\arctan \left( \frac{ 1-\sigma_1 }{ t_1 } \right)
    +t_2\arctan \left( \frac{ 1-\sigma_2 }{ t_2 } \right) 
    -(t_1+t_2)\arctan \left( \frac{ 1-\sigma_1-\sigma_2 }{ t_1+t_2 } \right) 
    \biggr) \\
   &\ge1.
  \end{split}
 \end{align}
 By \eqref{4441} and \eqref{4442}, we get
 \begin{align*}
  &\frac{\zeta(s_1)\zeta(s_2)}{\zeta(s_1+s_2)}
  \asymp\frac{ \alpha_1 }{ \alpha_2 } 
  \gg t^{1/2}. \qedhere
 \end{align*}
\end{proof}

\begin{lem} \label{4443}
 Let $y_1>y_2>0$. 
 For $s\in\mathbb{C}$ with $-y_1<\sigma<-y_2$ and sufficiently large $t$, we have 
 \[
  \widetilde\zeta(s)
  = \prod_{j=1}^r \zeta (a_{j}s) (1+O(t^{-1/2})).
 \]
\end{lem}
\begin{proof}
 By applying the previous lemma repeatedly, we find the result. 
\end{proof}

\begin{lem} \label{4444}
 For $b>0$, we have 
 \begin{align*}
  \Im \int_{-y+iT}^{-y+it_0} \frac{ \zeta'(bs) }{ \zeta(bs) } ds
  =T \log \frac{ bT }{ 2\pi } -T +O(\log T). 
 \end{align*}
\end{lem}
\begin{proof}
 By the functional equation, we have
 \begin{align*}
  \int_{-y+iT}^{-y+it_0} \frac{ \zeta'(s) }{ \zeta(s) } ds
  &=[\log \zeta(s)]_{-y+iT}^{-y+it_0} \\
  &=\left[ \log\left(2^s\pi^{s-1}\sin\left(\frac{\pi s}{2}\right)\Gamma(1-s)\zeta(1-s)\right) \right]_{-y+iT}^{-y+it_0}. 
 \end{align*} 
 Here we have
 \begin{align*}
  \Im \left[ \log\left(2^s \pi^{s-1} \right) \right]_{-y+iT}^{-y+it_0}
  =-T \log 2\pi +O(1)
 \end{align*}
 and
 \begin{align*}
  \Im \left[ \log\sin \left( \frac{ \pi s }{ 2 } \right) \right]_{-y+iT}^{-y+it_0}
  &=-\Im \left( \log \left( -\frac{ 1 }{ 2i } \exp \left( \frac{ -i \pi (-y+iT) }{ 2 } \right) \right) \right) +O(1) \\
  &=O(1). 
 \end{align*}
 By the Stirling formula and \cite[Section 9.4]{Tit86}, we also have 
 \begin{align*}
  \Im \left[ \log \Gamma(1-s) \right]_{-y+iT}^{-y+it_0}
  &= -\Im \left( \log (e^{iT} (1+y-iT)^{1/2-y-iT}) \right) +O(1) \\
  &=-T+T\log T +O(1)
 \end{align*}
and
 \begin{align*}
  \Im \left[ \log\zeta\left(1-s\right) \right]_{-y+iT}^{-y+it_0}
  =O(\log T).
 \end{align*}
 By changing variable, we obtain the result. 
\end{proof}

\begin{lem} \label{mo3c}
 We have 
 \begin{align*}
  \Im \int_{R_4} \frac{ G'(s) }{ G(s) } ds
  =
  \begin{cases}
   \displaystyle{ T \sum_{j=1}^r a_j\log \frac{ a_jT }{ 2\pi e} +T\log M +O(\log T)} &(a=0), \\
   \displaystyle{ T \sum_{j=1}^r a_j\log \frac{ a_jT }{ 2\pi e} +O(\log T)} &(a\ne0). 
  \end{cases}
 \end{align*}
\end{lem}
\begin{proof}
 By Lemmas \ref{4443} and \ref{4444}, we have 
 \begin{align*}
  \Im \int_{R_4} \frac{ \widetilde\zeta'(s) }{ \widetilde\zeta(s) } ds 
  &=\Im \sum_{j=1}^r \int_{R_4} a_j \frac{ \zeta'(a_js) }{ \zeta(a_js) } ds 
   +[\arg (1+O(t^{-1/2}))]_{R_4} \\
  &=T \sum_{j=1}^r a_j\log \frac{ a_jT }{ 2\pi e} +O(\log T). 
 \end{align*}
 When $a=0$, by \eqref{note}, we have 
 \begin{align*}
  \Im \int_{R_4} \frac{ G'(s) }{ G(s) } ds 
  &=T \sum_{j=1}^r a_j\log \frac{ a_jT }{ 2\pi e} +T\log M +O(\log T). 
 \end{align*}
 When $a\ne 0$, by the functional equation, the Stirling formula, and \eqref{note}, 
 there exists $t'>t_0$ such that
 \begin{align*}
  \frac{ G'(s) }{ G(s) }
  &=\frac{ \widetilde\zeta'(s) }{ \widetilde\zeta(s) } \cdot \frac{ 1 }{ 1 -a/\widetilde\zeta(s) } \\
  &=\frac{ \widetilde\zeta'(s) }{ \widetilde\zeta(s) } (1+O(t^{-r/2-Ay})) 
 \end{align*}
 holds for $t\ge t'$. 
 We note that $\zeta'(s)/\zeta(s)\ll \log t$ by the functional equation and \cite[Lemma 2.1]{Ono17}, we have 
 \begin{align*}
  \Im \int_{R_4} \frac{ G'(s) }{ G(s) } ds 
  =T \sum_{j=1}^r a_j\log \frac{ a_jT }{ 2\pi e} +O(\log T). 
 \end{align*} 
 This finishes the proof. 
\end{proof}

\subsection{Proof of Theorems \ref{main2+} and \ref{main3}}
\begin{proof}[Proof of Theorem \ref{main2+}]
 By the functional equation, the Stirling formula, and Lemma \ref{4443}, we have 
 \[
  \widetilde\zeta(s)
  \gg t^{r/2-A	\sigma}
 \]
 for large $t$. 
 Then we have the result. 
\end{proof}

\begin{proof}[Proof of Theorem \ref{main3}]
 We note that
 \[
  \int_{R_1} \frac{ G'(s) }{ G(s) } ds
  =O(1)
 \]
 holds. 
 By Lemma \ref{22}, we have
 \begin{align*}
  \Im \int_{R_2} \frac{ G'(s) }{ G(s) } ds 
  &=\arg \left( G(x+iT) \right)
   -\arg \left( G(x+it_0) \right)=O(1). 
 \end{align*}
 Then, by Lemmas \ref{mo3b} and \ref{mo3c}, we find the result. 
\end{proof}

\section{Proof of Theorems \ref{main4} and \ref{main5}}
\begin{lem}
 We have
 \begin{align*}
  \lim_{X\rightarrow\infty} \int_{X+it_0}^{X+iT} \log |G(s)| ds
  =0.
 \end{align*}
\end{lem}
\begin{proof}
 By the definition of $G(s)$, we have
 \begin{align*}
  \int_{X+it_0}^{X+iT} \log |G(s)| ds
  =
  \begin{cases}
   \displaystyle{ \int_{X+it_0}^{X+iT} \log \left| 1+O(c^\sigma) \right| ds} &(a=0), \\
   \vspace{-0.5em} \\
   \displaystyle{ \int_{X+it_0}^{X+iT} \log \left| 1+O(M^\sigma) \right| ds} &(a\ne0), 
  \end{cases}
 \end{align*}
 where $c$ is the constant in Lemma \ref{22}. 
\end{proof}

\begin{lem} \label{52}
 We have
 \begin{align*}
  \int_{x+it_0}^{x+iT} \log |G(s)| ds 
  =O(1). 
 \end{align*}
\end{lem}
\begin{proof}
 By Cauchy's theorem, we have
 \begin{align*}
  \int_{x+it_0}^{x+iT} \log |G(s)| ds 
  =\int_{x+it_0}^{\infty+it_0} \log |G(s)| ds
   -\int_{x+iT}^{\infty+iT} \log |G(s)| ds. 
 \end{align*}
 Since $\log (1+\xi)=O(|\xi|)$ for $\xi>-1/2$, we have 
 \begin{align*}
  \int_{x+it_0}^{\infty+it_0} \log |G(s)| ds
  &=\int_{x+it_0}^{\infty+it_0} \log \left| 1+O(\xi'^\sigma) \right| ds \\
  &=\int_{x+it_0}^{\infty+it_0} O( \xi'^\sigma ) ds \\
  &=O(1), 
 \end{align*}
 where $\xi'=c$ if $a=0$, otherwise $M$. 
 Similarly, we have
 \begin{align*}
  \int_{x+iT}^{\infty+iT} \log |G(s)| ds
  =O(1).  
 \end{align*}
 Thus we get the result. 
\end{proof}

\begin{lem} \label{rz}
 Let $y>0$. 
 For large $T$, we have 
 \begin{align*}
  \int_{t_0}^{T} \log|\zeta(-y+it)| dt
  &=\left( \frac{ 1 }{ 2 } +y \right) T\log T +O(T). 
 \end{align*}
\end{lem}
\begin{proof}
 Let $\chi(s):=2^s\pi^{s-1}\sin\left(\pi s/2 \right)\Gamma(1-s)$. 
 Then we easily find 
 \[
  \log |\chi(s)|=\left( \frac{ 1 }{ 2 } -\sigma \right) \log\left| \frac{ t }{ 2\pi } \right| +O\left( \frac{ 1 }{ t } \right) 
 \]
 by the Stirling formula for fixed $\sigma$ and sufficiently large $t$. 
 Since $\log |\zeta(1+y-it)|=O(1)$, we have
 \begin{align*}
  \int_{t_0}^{T} \log|\zeta(-y+it)| dt
  &=\int_{t_0}^{T} \log|\chi(-y+it)| dt 
   +\int_{t_0}^{T} \log|\zeta(1+y-it)| dt \\
  &=\left( \frac{ 1 }{ 2 } +y \right) T\log T +O(T). \qedhere  
 \end{align*}
\end{proof}

\begin{lem} \label{54}
 Let $y>0$. 
 For large $T$, we have 
 \[
  \int_{t_0}^{T} \log|G(-y+it)| dt
  =\sum_{j=1}^r \left( \frac{ 1 }{ 2 } +a_j y \right) T\log T   
   +O(T).  
 \]
\end{lem}
\begin{proof}
 When $a=0$, by Lemma \ref{4443}, we have
 \begin{align*}
  \int_{t_0}^{T} \log|G(-y+it)| dt
  &=\sum_{j=1}^r \int_{t_0}^{T} \log |\zeta (a_{j}(-y+it))| dt +O(T) \\
  &=\sum_{j=1}^r \frac{ 1 }{ a_j } \int_{a_j t_0}^{a_j T} \log |\zeta (-a_j y+it)| dt 
   +O(T). 
 \end{align*}
 By Lemma \ref{rz}, we have the result. 
 On the other hand, when $a\ne0$, we have 
 \begin{align*}
  \log | G(s) |
  &=\log |\widetilde\zeta(s)-a| +O(1) \\
  &=\log |\widetilde\zeta(s)| +\log \left|1- \frac{ a }{\widetilde\zeta(s) }\right| +O(1).
 \end{align*}
 By the similar argument in Lemma \ref{mo3c}, we also find the result.  
\end{proof}

\begin{prop} \label{prop55}
 Let $y>0$. 
 For large $T$, we have 
 \begin{align*}
  2\pi \sum_{\substack{ -y < \beta_a \\ 0<\gamma_a <T }} 
  \left( \beta_a +y \right)
  =\sum_{j=1}^{r} \left( \frac{ 1 }{ 2 }+a_jy \right) T\log T +O(T). 
 \end{align*}
\end{prop}
\begin{proof}
 Applying Littlewood's lemma (for details, see \cite[Section 3.8]{Tit39}) 
 to the function $G(s)$ in the rectangle 
 $\{ s \mid -y \le \sigma \le x, t_0 \le t \le T \}$, we have 
 \begin{align*}
  2\pi \sum_{\substack{ -y < \beta_a \\ 0<\gamma_a <T }} 
   ( \beta_a +y ) 
  &=-\int_{t_0}^{T} \log|G(x+it)| dt 
   +\int_{t_0}^{T} \log|G(-y+it)| dt \\
  &\quad +\int_{-y}^{x} \arg G(\sigma+iT) d\sigma 
   -\int_{-y}^{x} \arg G(\sigma+it_0) d\sigma.
 \end{align*}
 We easily see
 \begin{align} \label{eq1}
  \int_{-y}^{x} \arg G(\sigma+it_0) d\sigma=O(1).  
 \end{align}
 By \eqref{sekibunmae}, we have 
 \begin{align} \label{eq2}
  \int_{-y}^{x} \arg G(\sigma+iT) d\sigma
  &\ll \log T. 
 \end{align}
 Then, by Lemmas \ref{52} and \ref{54}, we obtain the result. 
\end{proof}

\begin{proof}[Proof of Theorem \ref{main4}]
 We note that 
 \begin{align*}
  2\pi \sum_{\substack{ -y < \beta_a \\ 0<\gamma_a <T }} 
  \left( \beta_a -\frac{ 1 }{ 2 } \right)
  &=  2\pi \sum_{\substack{ -y < \beta_a \\ 0<\gamma_a <T }} 
  \left( \beta_a +y \right)
  -2\pi \left( y+\frac{ 1 }{ 2 }\right) N_{y}(a;T), \\
  2\pi \sum_{\substack{ -y < \beta_a \\ 0<\gamma_a <T }} 
  \left( \beta_a -\frac{ r }{ 2A } \right)
  &=  2\pi \sum_{\substack{ -y < \beta_a \\ 0<\gamma_a <T }} 
  \left( \beta_a +y \right)
  -2\pi \left( y+\frac{ r }{ 2A }\right) N_{y}(a;T). 
 \end{align*}
 By Theorem \ref{main3} and Proposition \ref{prop55}, we obtain the result. 
\end{proof}

\begin{lem}
 For $\sigma\ge 1/2$ and $0\le\alpha\le2$, we have 
 \[
  \int_{t_0}^{T} |\zeta(\sigma+it)|^\alpha dt \ll T\log T. 
 \]
\end{lem}
\begin{proof}
 The lemma is trivial when $\alpha=0$.
 When $\alpha=2$, see \cite[Section 7.2]{Tit86}. 
 As for the case $0<\alpha<2$, let $\beta$ satisfy $\alpha/2+1/\beta=1$. 
 By H\"{o}lder's inequality, we have 
 \begin{align*}
  \int_{t_0}^{T} |\zeta(\sigma+it)|^\alpha dt 
  &\le \left( \int_{t_0}^{T} |\zeta(\sigma+it)|^2 dt \right)^{\alpha/2} \left( \int_{t_0}^{T} dt \right)^{1/\beta} \\
  &\ll (T\log T)^{\alpha/2} T^{1/\beta}. \qedhere
 \end{align*} 
\end{proof}

\begin{prop} \label{prop56} 
 Let $y_3=1/(2a_r)$. 
 We have
 \[
  2\pi \sum_{\substack{ y_3 < \beta_a \\ 0<\gamma_a <T }} 
  \left( \beta_a -y_3 \right)
  \ll T\log\log T. 
 \]
\end{prop}
\begin{proof}
 We prove the proposition only for $a\ne0$.
 We can prove the case $a=0$ similarly. 
 Since $\sum_{j} x_j \le \prod_{j} (x_j+1)$ for $x_j \ge0$, there exists a constant $C$ such that 
 \begin{align*}
  &\int_{t_0}^{T} \log |\widetilde\zeta(y_3+it)-a| dt \\
  &\le \int_{t_0}^{T} \log 
   \left( 
    C \sum_{l=1}^{r} \sum_{P_1,\dots,P_l} \prod_{k=1}^{l} 
    \left| \zeta \left( \sum_{a_j\in P_k} a_j(y_3+it) \right) \right| +|a| 
   \right) dt \\
  &\le \sum_{l=1}^{r} \sum_{P_1,\dots,P_l} \frac{ 1 }{ u } \int_{t_0}^{T} \log 
   \left( \left(
    C \prod_{k=1}^{l} \left| \zeta \left( \sum_{a_j\in P_k} a_j(y_3+it) \right) \right| +1 
   \right)^u \right) dt 
    +\int_{t_0}^{T} \log(|a|+1) dt 
 \end{align*}
 for $0<u<2/r$. 
 Since $(X+1)^u\le X^u+1$, we have 
 \begin{align*}
  &\int_{t_0}^{T} \log |\widetilde\zeta(y_3+it)-a| dt \\
  &\le \sum_{l=1}^{r} \sum_{P_1,\dots,P_l} \frac{ 1 }{ u } \int_{t_0}^{T} \log 
   \left( 
    C^u \prod_{k=1}^{l} \left| \zeta \left( \sum_{a_j\in P_k} a_j (y_3+it) \right) \right|^u +1 
   \right) dt +O(T). 
 \end{align*}
 By Jensen's inequality, we have
 \begin{align*}
  &\int_{t_0}^{T} \log |\widetilde\zeta(y_3+it)-a| dt \\
  &\le \sum_{l=1}^{r} \sum_{P_1,\dots,P_l} \frac{ T-t_0 }{ u } \log 
   \left( 
    \frac{ 1 }{ T-t_0 } \int_{t_0}^{T} \left( C^u \prod_{k=1}^{l} \left| \zeta \left( \sum_{a_j\in P_k} a_j(y_3+it) \right) \right|^u +1 
   \right) dt \right) +O(T).
 \end{align*} 
 Since $0<ul<2$, by H\"{o}lder's inequality and the previous lemma, we have 
 \begin{align*}
  \int_{t_0}^{T} \prod_{k=1}^{l} \left| \zeta \left( \sum_{a_j\in P_k} a_j(y_3+it) \right) \right|^u dt
  &\le \prod_{k=1}^{l} \left( \int_{t_0}^{T} \left| \zeta \left( \sum_{a_j\in P_k} a_j(y_3+it) \right) \right|^{ul} dt \right)^{1/l} \\
  &\le T\log T. 
 \end{align*}
 Thus we get
 \[
  \int_{t_0}^{T} \log |\widetilde\zeta(y_3+it)-a| dt 
  \ll T\log\log T. 
 \]
 By Littlewood's lemma, we have the result.  
\end{proof}

\begin{proof}[Proof of Theorem \ref{main5}]
 By Proposition \ref{prop56}, we have 
 \begin{align*}
  \sum_{\substack{ \beta_a-y_3>\delta \\ 0<\gamma_a <T }} 1
  &\le \sum_{\substack{ \beta_a-y_3>\delta \\ 0<\gamma_a <T }} \frac{ \beta_a-y_3 }{ \delta } \\
  &\le \frac{ 1 }{ \delta } \sum_{\substack{ y_3 < \beta_a \\ 0<\gamma_a <T }} (\beta_a-y_3) \\
  &\ll \frac{ T }{ \delta } \log \log T.
 \end{align*}
 Putting $\delta=(\log \log T)^2/\log T$, we obtain the result. 
\end{proof}


\end{document}